\documentclass[12pt, dvipdfmx]{article}
\usepackage[mathscr]{eucal}
\usepackage{amssymb}
\usepackage{latexsym}
\usepackage{amsthm}
\usepackage{amsmath}
\usepackage[dvips]{graphicx}
\usepackage{psfrag}
\usepackage{a4wide}
\usepackage{tikz-cd}
\usepackage{color}

\newtheorem{thm}{Theorem}[section]

\newtheorem{lem}[thm]{Lemma}
\newtheorem{rem}[thm]{Remark}
\newtheorem{ex}[thm]{Example}

\newtheorem{prop}[thm]{Proposition}

\newtheorem{acknowledgment}{Acknowledgment}

\numberwithin{equation}{section}
\newcommand{\C}{{\mathbb C}}
\newcommand{\D}{{\mathbb D}}
\newcommand{\B}{{\mathbb B}}

\newcommand{\cS}{{\mathcal S}}
\newcommand{\cR}{{\mathcal R}}

\newcommand{\cSB}{{\mathcal S}_{B}}

\newcommand{\cH}{{\mathcal H}}
\newcommand{\cK}{{\mathcal K}}

\newcommand{\ran}{\operatorname{ran}}

\newcommand{\la}{\langle}
\newcommand{\ra}{\rangle}

\newcommand{\lam}{\lambda}

\newcommand{\w}{\omega}

\newcommand{\blam}{\boldsymbol{\lam}}
\newcommand{\bmu}{\boldsymbol{\mu}}
\newcommand{\bnu}{\boldsymbol{\nu}}

\newcommand{\bw}{\boldsymbol{\omega}}
\newcommand{\bz}{\boldsymbol{z}}
\newcommand{\by}{\boldsymbol{y}}
\newcommand{\bx}{\boldsymbol{x}}
\newcommand{\ba}{\boldsymbol{a}}
\newcommand{\bb}{\boldsymbol{b}}

\newcommand{\0}{\boldsymbol{0}}
\newcommand{\bu}{\boldsymbol{u}}

\newcommand{\bbw}{\boldsymbol{w}}

\begin{document}

\title{Indefinite structure of the Bergman kernel\\ on the open unit disk
}
\author{
{\sc Kenta KOJIN}\\
[1ex]
{\small Japan Coast Guard Academy, Kure 737-8512, Japan}\\
{\small
{\it E-mail address}: {\tt k-koujin@jcga.ac.jp}}\\
[1ex]
{\sc Shuhei KUWAHARA}\\
[1ex]
{\small Sapporo Seishu High School, Sapporo 064-0916, Japan} \\
{\small 
{\it E-mail address}: {\tt s.kuwahara@sapporoseishu.ed.jp}}
\\
[1ex]
and\\
[1ex]
{\sc Michio SETO}
\\
[1ex]
{\small National Defense Academy of Japan,  
Yokosuka 239-8686, Japan} \\
{\small 
{\it E-mail address}: {\tt mseto@nda.ac.jp}}
}

\date{}

\maketitle
\begin{abstract}
In this paper, 
we focus on an indefinite structure lying behind the Bergman kernel on the open unit disk. 
In particular, an invariant distance, birational maps 
and an indefinite kernel are constructed from the Bergman kernel, 
and we deal with their interaction.  
\end{abstract}

\begin{center}
keywords: Bergman kernel, Kre\u{\i}n space
\end{center}

\section{Introduction}
Let $\D$ denote the open unit disk in the complex plane $\C$, and let $k$ denote the Bergman kernel on $\D$, 
that is, we set
\[
k(z,\lam)=\dfrac{1}{(1-\overline{\lam}z)^2}\quad (z,\lam \in \D).
\]
Operator theory induced by the Bergman kernel has been studied in a number of papers, 
for example, 
Barbian~\cite{Barbian}, Chailos~\cite{Chailos1, Chailos2}, McCullough--Richter~\cite{MR}, Yang--Zhu~\cite{YZ}. 
The purpose of this paper is to give a new point of view to functional analysis related to the Bergman kernel. 
In order to state our main result, we need a map and a matrix. 
Let $\Phi$ denote the nonlinear transformation defined as follows: 
\[
\Phi:\D \to \C^2,\quad \lam\mapsto 
\begin{pmatrix}
\sqrt{2}\lam\\
\lam^2
\end{pmatrix}.
\]
For any $(\lam_1,\lam_2)^{\top}$ and $(\mu_1,\mu_2)^{\top}$ in $\C^2$, we set 
\[
\la (\lam_1,\lam_2)^{\top},(\mu_1,\mu_2)^{\top} \ra_{\cK}=\lam_1\overline{\mu_1}-\lam_2\overline{\mu_2}
\]
and consider the indefinite inner product space $\cK=(\C^2,\la \cdot,\cdot\ra_{\cK})$.  
Then, since 
\[
1-\la \Phi(\lam), \Phi(\mu) \ra_{\cK}=1-(2\lam\overline{\mu}-\lam^2\overline{\mu^2})
=(1-\lam\overline{\mu})^2, 
\]
we have that
\[
k(z,\lam)=\dfrac{1}{1-\la \Phi(z), \Phi(\lam) \ra_{\cK}}.
\]
Moreover, we set 
\[
M_{\Phi}(\lam_1,\lam_2;\mu_1,\mu_2)=
\begin{pmatrix}
\dfrac{1-\la \Phi(\mu_1), \Phi(\mu_1)\ra_{\cK}}{1-\la \Phi(\lam_1),\Phi(\lam_1) \ra_{\cK}} 
& \dfrac{1-\la \Phi(\mu_1), \Phi(\mu_2) \ra_{\cK}}{1-\la \Phi(\lam_1),\Phi(\lam_2) \ra_{\cK}}\\[0.5cm]
\dfrac{1-\la \Phi(\mu_2), \Phi(\mu_1) \ra_{\cK}}{1-\la \Phi(\lam_2), \Phi(\lam_1) \ra_{\cK}}
& \dfrac{1-\la \Phi(\mu_2), \Phi(\mu_2)\ra_{\cK}}{1-\la \Phi(\lam_2),\Phi(\lam_2)\ra_{\cK}}
\end{pmatrix}.
\]
As a part of the main result (Theorem \ref{thm:4-1}), 
under the condition that $M_{\Phi}(\lam_1,\lam_2;\w_1,\w_2)$ is positive semi-definite,  
we prove not only that the two-point Pick problem $\lam_j\mapsto \omega_j$ $(j=1,2)$ is solvable, 
but also that an indefinite two-point Pick problem is solvable as follows: 
there exists a rational map $F(z_1,z_2)=(f_1(z_1,z_2),f_2(z_1,z_2))$ 
such that 
\begin{enumerate}
\item $F(\Phi(\lam_j))=\Phi(\w_j)$ for any $j=1,2$, 
\item $|f_1(z_1,z_2)|^2-|f_2(z_1,z_2)|^2<1$
if $|z_1|^2-|z_2|^2<1$. 
\end{enumerate}
We note that the Bergman kernel does not have the two-point Pick property.

This paper is organized as follows. 
In Section 2, an invariant distance on $\D$ is introduced and 
a Schwarz--Pick type inequality is discussed.  
In Section 3, we deal with birational self-maps of a hyperbolic ball. 
In Section 4, it is proved that 
the Bergman space is isometrically embedded into a reproducing kernel Kre\u{\i}n space.  
In Section 5, we discuss structure of Pick matrices and prove the main theorem. 

\section{Invariant distance}
We set 
\[
\rho(\lam,\mu)=\sqrt{1-\dfrac{|k(\lam,\mu)|^2}{k(\lam,\lam)k(\mu,\mu)}}\quad (\lam,\mu \in \D).
\]
Then, $\rho$ is a distance on $\D$ by Lemma 9.9 of Agler--McCarthy~\cite{AM}. 
Moreover, a direct calculation shows that
\begin{equation}\label{eq:1-1}
\rho(\lam,\mu)=\sqrt{2\left|\dfrac{\lam-\mu}{1-\overline{\mu}\lam}\right|^2
-\left|\dfrac{\lam-\mu}{1-\overline{\mu}\lam}\right|^4}.
\end{equation}
Hence $\rho$ is invariant under the action of the automorphism group of $\D$.
Let $\cS$ denote the Schur class, that is, we write
\[
\cS=\left\{f \in \operatorname{Hol}(\D): \sup_{\lam \in \D}|f(\lam)|\leq 1\right\}.
\]

\begin{prop}\label{prop:1-1}
Let $f$ be a function in $\cS$ that is not a constant with absolute value one. 
Then, for any two points $z$ and $w$ in $\D$,
\[
\rho(f(z),f(w))\leq \rho(z,w).
\] 
Moreover, $f$ is isometric with respect to $\rho$ if and only if $f$ is an automorphism of $\D$. 
\end{prop}

\begin{proof}
We set 
\[
X=\left|\dfrac{z-w}{1-\overline{w}z}\right|\quad \mbox{and}\quad X_f=\left|\dfrac{f(z)-f(w)}{1-\overline{f(w)}f(z)}\right|.
\]
Then, by (\ref{eq:1-1}), we have
\begin{align*}
\rho(z,w)^2-\rho(f(z),f(w))^2
&=(2X^2-X^4)-(2X_f^2-X_f^4)\\
&=(X^2-X_f^2)(2-X^2-X_f^2).
\end{align*}
Since the Schwarz--Pick inequality means that $X_f\leq X$, 
we conclude the proof. 
\end{proof}

\begin{prop}\label{lem:1-2}
Let $\lam_1,\lam_2, \mu_1$ and $\mu_2$ be points in $\D$. 
Then, 
$M_{\Phi}(\lam_1,\lam_2;\mu_1,\mu_2)$ is positive semi-definite 
if and only if $\rho(\mu_1,\mu_2) \leq \rho (\lam_1,\lam_2)$. 
\end{prop}

\begin{proof}
We set $M_{\Phi}=M_{\Phi}(\lam_1,\lam_2;\mu_1,\mu_2)$. 
Since the diagonal entries of $M_{\Phi}$ are nonnegative,  
it suffices to show that 
$\det M_{\Phi}\geq 0$ is equivalent to $\rho(\mu_1,\mu_2) \leq \rho (\lam_1,\lam_2)$. 
First, by definition, we have the following identity:
\begin{equation}\label{eq:1-2}
\dfrac{(1-\la \Phi(\lam),\Phi(\lam) \ra_{\cK})(1-\la \Phi(\mu),\Phi(\mu) \ra_{\cK})}
{|1-\la \Phi(\lam),\Phi(\mu) \ra_{\cK}|^2}
=1-\rho(\lam,\mu)^2
\end{equation}
for any $\lam$ and $\mu$ in $\D$.
Then, we have 
\begin{align*}
\det M_{\Phi}
&=\dfrac{(1-\la \Phi(\mu_1), \Phi(\mu_1)\ra_{\cK})(1-\la \Phi(\mu_2), \Phi(\mu_2)\ra_{\cK})}
{(1-\la \Phi(\lam_1),\Phi(\lam_1) \ra_{\cK})(1-\la \Phi(\lam_2),\Phi(\lam_2)\ra_{\cK})} 
-\left| \dfrac{1-\la \Phi(\mu_1), \Phi(\mu_2) \ra_{\cK}}{1-\la \Phi(\lam_1),\Phi(\lam_2) \ra_{\cK}} \right|^2\\[0.3cm]
&=
\dfrac{\left|1-\la \Phi(\mu_1), \Phi(\mu_2)\ra_{\cK} \right|^2\left(
\rho(\lam_1,\lam_2)^2-\rho(\mu_1,\mu_2)^2
\right)}
{(1-\la \Phi(\lam_1),\Phi(\lam_1) \ra_{\cK})(1-\la \Phi(\lam_2),\Phi(\lam_2)\ra_{\cK})}.
\end{align*}
This concludes the proof. 
\end{proof}

The following class denoted as $\cS_B$ is suggested by Proposition \ref{lem:1-2}:
\begin{equation*}\label{eq:2-2} 
\cS_B=
\left\{f\in \operatorname{Hol}(\D):
\dfrac{1-\la \Phi(f(z)), \Phi(f(\lam))\ra_{\cK}}{1-\la \Phi(z), \Phi(\lam)\ra_{\cK}}
\ \mbox{is positive semi-definite in $\D^2$}
\right\}.
\end{equation*}
A characterization of $\cSB$ is given in the following theorem.  
\begin{thm}\label{theorem:2-1}
$\cSB=\cS \cup \cS^{-1}$, where we set 
\[
\cS^{-1}= \left\{f\in\operatorname{Hol}(\D) : \inf_{\lam \in \D} |f(\lam)|\geq 1\right\}.
\]
\end{thm}

\begin{proof}
First, we assume that $f$ belongs to $\cSB$.
If $|f(\lam)|=1$ for some $\lam$ in $\D$, 
then we have
\[\det
\begin{pmatrix}
\dfrac{1-\la \Phi(f(\lam)), \Phi(f(\lam))\ra_{\cK}}{1-\la \Phi(\lam), \Phi(\lam)\ra_{\cK}} 
& \dfrac{1-\la \Phi(f(\lam)), \Phi(f(\mu))\ra_{\cK}}{1-\la \Phi(\lam), \Phi(\mu)\ra_{\cK}}\\[0.5cm]
\dfrac{1-\la \Phi(f(\mu)), \Phi(f(\lam))\ra_{\cK}}{1-\la \Phi(\mu), \Phi(\lam)\ra_{\cK}} &
\dfrac{1-\la \Phi(f(\mu)), \Phi(f(\mu))\ra_{\cK}}{1-\la \Phi(\mu), \Phi(\mu)\ra_{\cK}}
\end{pmatrix}
=-\left|\dfrac{(1-f(\lam)\overline{f(\mu)})^2}{(1-\lam\overline{\mu})^2}\right|^2\geq 0
\]
for any $\mu$ in $\D$. 
This implies that $f(\D)$ does not intersect with the unit circle whenever $f$ is not a unimodular constant. 
Hence $f$ is in $\cS \cup \cS^{-1}$.
Next, we suppose that $f$ is a function in $\cS \cup \cS^{-1}$. 
If $f$ is in $\cS$, then, it follows from Schur's theorem that 
\[
\dfrac{1-\la \Phi(f(z)), \Phi(f(\lam))\ra_{\cK}}{1-\la \Phi(z), \Phi(\lam)\ra_{\cK}}
=\left(\dfrac{1-\overline{f(\lam)}f(z)}{1-\overline{\lam}z}\right)^2
\]
is positive semi-definite. Hence, $f$ belongs to $\cSB$. 
If $f$ is in $\cS^{-1}$, that is, $1/f$ is in $\cS$, then we have
\[
\dfrac{1-\la \Phi(f(z)), \Phi(f(\lam))\ra_{\cK}}{1-\la \Phi(z), \Phi(\lam)\ra_{\cK}}
=\left(\dfrac{1-\overline{(1/f(\lam))}(1/f(z))}{1-\overline{\lam}z}
\overline{f(\lam)}f(z)\right)^2
\]
is positive semi-definite by Schur's theorem. 
Hence, $f$ belongs to $\cSB$. 
\end{proof}

\section{Birational maps}

In this section, we shall consider the unit ball of $\cK=(\C^2,\la \cdot,\cdot\ra_{\cK})$.  
In what follows, each element in $\cK$ is denoted such as $\bz$ in bold type. 
Let $\Omega$ be the unit ball in $\cK$, that is, we set
\begin{align*}
\Omega
&=\{\bz\in \cK: \la\bz,\bz\ra_{\cK}<1\}\\
&=\{\bz=(z_1,z_2)^{\top}\in \C^2: |z_1|^2-|z_2|^2<1\}.
\end{align*}
A vector $\ba$ in $\cK$ is said to be neutral if $\la \ba ,\ba\ra_{\cK}=0$. 
Let $\ba$ be a nonneutral vector in $\Omega$. 
Then, for any $\bz$ in $\Omega$, we set
\[
P_{\ba}\bz=\frac{\la \bz,\ba \ra_{\cK}}{\la \ba,\ba\ra_{\cK}}\ba.
\]
It is not difficult to see that $P_{\ba}$ is idempotent and self-adjoint with respect to the inner product of $\cK$. 
Moreover, we set $s_{\ba}=(1-\la \ba,\ba \ra_{\cK})^{1/2}$, $Q_{\ba}=I-P_{\ba}$ and 
\[
\varphi_{\ba}(\bz)=\frac{\ba-P_{\ba}\bz-s_{\ba}Q_{\ba}\bz}{1-\la \bz,\ba \ra_{\cK}}
\]
(cf. Section 2.2 in Rudin~\cite{Rudin}, and see also K.D.--Kojin--Seto~\cite{KDKS}). 
Although $\0=(0,0)^{\top}$ is a neutral vector, as an exception, 
it is convenient to define $P_{\0}\bz=\0$ and $\varphi_{\0}(\bz)=-\bz$. 
We set
\[
\Omega_{\ba}=\Omega \setminus \{\bz\in\Omega:\la \bz,\ba\ra_{\cK}=1\}. 
\] 
A matrix $T$ is said to be $\sharp$-unitary if $T^{\sharp}=T^{-1}$ where $T^{\sharp}$ 
denotes the adjoint matrix of $T$ with respect to the inner product of $\cK$. 

\begin{prop}\label{prop:1-3}
Let $\ba$ be a nonneutral vector in $\Omega$. 
Then, 
\begin{enumerate}
\item 
$\varphi_{\ba}(\0)=\ba$ and $\varphi_{\ba}(\ba)=\0$.
\item 
For any $\bz$ and $\bbw$ in $\Omega_{\ba}$, 
\[
1-\la \varphi_{\ba}(\bz),\varphi_{\ba}(\bbw) \ra_{\cK}=\frac{1-\la \ba, \ba\ra_{\cK}}
{(1-\la \bz,\ba \ra_{\cK})(1-\la \ba,\bbw \ra_{\cK})}(1-\la \bz,\bbw \ra_{\cK}).
\]
In particular, for any $\bz$ in $\Omega_{\ba}$, 
\[
1-\la \varphi_{\ba}(\bz),\varphi_{\ba}(\bz) \ra_{\cK}=\frac{1-\la \ba, \ba\ra_{\cK}}
{|1-\la \bz,\ba \ra_{\cK}|^2}(1-\la \bz,\bz \ra_{\cK}).
\]
\item For any $\bz$ and $\bbw$ in $\Omega_{\ba}$, 
\[
1-\la \varphi_{\ba}(\bz),\bbw \ra_{\cK}=\frac{1-\la \ba,\bbw\ra_{\cK}}{1-\la \bz,\ba \ra_{\cK}}(1-\la \bz, \varphi_{\ba}(\bbw) \ra_{\cK}).
\]
\item For any $\bz$ in $\Omega_{\ba}$, 
$\varphi_{\ba}(\bz)$ belongs to $\Omega_{\ba}$
and $\varphi_{\ba}(\varphi_{\ba}(\bz))=\bz$. In particular, $\varphi_{\ba}|_{\Omega_{\ba}}$ is an automorphism of $\Omega_{\ba}$. 
\item For any $\sharp$-unitary matrix $T$, 
$T\varphi_{\ba}T^{\sharp}=\varphi_{T\ba}$. 
\item For any $z$ and $w$ in $\D$, 
\[
\rho(z,w)^2=\la \varphi_{\Phi(w)}(\Phi(z)), \varphi_{\Phi(w)}(\Phi(z)) \ra_{\cK}. 
\]
\end{enumerate}
\end{prop}
\begin{proof}\rm 
Routine calculation shows (i), (ii) and (iii). 
We shall show (iv). 
It follows from (ii) that $\varphi_{\ba}$ maps $\Omega_{\ba}$ to $\Omega$. Moreover, 
Since
\[
\la \varphi_{\ba}(\bz),\ba\ra_{\cK}
=\frac{\la \ba,\ba\ra_{\cK}-\la \bz,\ba\ra_{\cK}}{1-\la \bz,\ba \ra_{\cK}}, 
\]
we have that $\la \varphi_{\ba}(\bz),\ba\ra_{\cK}\neq 1$. 
Hence $\varphi_{\ba}|_{\Omega_{\ba}}$ is a self-map of $\Omega_{\ba}$. 
Further, direct calculation shows that $\varphi_{\ba}(\varphi_{\ba}(\bz))=\bz$.   
Thus, we have (iv). 
It is not difficult to see (v). 
By (\ref{eq:1-2}) and (ii), we have (vi).  
\end{proof}

The following fact might be known. However, we have not found appropriate literature, 
so that we include a proof here for the sake of readers.  

\begin{thm}\label{thm:2-2}
Let $\bb$ be any nonneutral vector in $\Omega_{\ba}$. Then, 
there exists a $\sharp$-unitary matrix $T$ such that 
$\varphi_{\bb}\circ \varphi_{\ba}=T \varphi_{\varphi_{\ba}(\bb)}$. 
\end{thm}

\begin{proof}
Since $\varphi_{\ba}$ is a linear fractional map, making some calculations,   
we obtain that the denominator of $\varphi_{\bb}(\varphi_{\ba}(\bz))$ is 
\[
(1-\la \ba,\bb\ra_{\cK})(1-\la \bz,\varphi_{\ba}(\bb)\ra_{\cK})
\]
and the numerator of $\varphi_{\bb}(\varphi_{\ba}(\bz))$ is a linear expression of $\bz$. 
By this observation, we have that 
\begin{align*}
\varphi_{\bb}\circ \varphi_{\ba}\circ \varphi_{\varphi_{\ba}(\bb)}(\bz)
&=\dfrac{\text{a linear expression of $\varphi_{\varphi_{\ba}(\bb)}(\bz)$}}{1-\la \varphi_{\varphi_{\ba}(\bb)}(\bz),\varphi_{\ba}(\bb) \ra_{\cK}}\\[0.3cm]
&=\dfrac{\dfrac{\text{a linear expression of $\bz$}}{1-\la \bz,\varphi_{\ba}(\bb) \ra}}
{1-\left\la \dfrac{\varphi_{\ba}(\bb)-P_{\varphi_{\ba}(\bb)}\bz-s_{\varphi_{\ba}(\bb)}Q_{\varphi_{\ba}(\bb)}\bz}{1-\la \bz,\varphi_{\ba}(\bb)\ra_{\cK}},\varphi_{\ba}(\bb) \right\ra_{\cK}}\\[0.3cm]
&=\dfrac{\text{a linear expression of $\bz$}}{(1-\la \bz,\varphi_{\ba}(\bb) \ra_{\cK})-(
\la \varphi_{\ba}(\bb),\varphi_{\ba}(\bb) \ra_{\cK}-\la \bz, \varphi_{\ba}(\bb) \ra_{\cK})}\\[0.3cm]
&=\dfrac{\text{a linear expression of $\bz$}}{1-\la \varphi_{\ba}(\bb),\varphi_{\ba}(\bb) \ra_{\cK}}.
\end{align*}
Hence, there exist a matrix $T$ and a vector $\bu$ such that 
\[
\varphi_{\bb}\circ \varphi_{\ba}\circ \varphi_{\varphi_{\ba}(\bb)}(\bz)=T\bz +\bu. 
\]
Further, 
it follows from (i) and (iv) of Proposition \ref{prop:1-3} that
\[
\varphi_{\bb}\circ \varphi_{\ba}\circ \varphi_{\varphi_{\ba}(\bb)}(\0)=\0,
\]
that is, $\bu=\0$. Moreover, by (iv) of Proposition \ref{prop:1-3}, we have that
\[
\varphi_{\bb}\circ \varphi_{\ba}\circ \varphi_{\varphi_{\ba}(\bb)}^{-1}(\bz)
=\varphi_{\bb}\circ \varphi_{\ba}\circ \varphi_{\varphi_{\ba}(\bb)}(\bz)=T\bz.
\]
Thus we have $\varphi_{\bb}\circ \varphi_{\ba}=T \varphi_{\varphi_{\ba}(\bb)}$.
Hence, applying (ii) and (iii) of Proposition \ref{prop:1-3} repeatedly, we have that 
\begin{align*}
1-\la T\bz,T\bz \ra_{\cK}
&=1- \la \varphi_{\bb}\circ \varphi_{\ba}\circ \varphi_{\varphi_{\ba}(\bb)}(\bz),
\varphi_{\bb}\circ \varphi_{\ba}\circ \varphi_{\varphi_{\ba}(\bb)}(\bz) \ra_{\cK}\\[0.3cm]
&=\frac{1-\la \bb,\bb\ra_{\cK}}{|1-\la \varphi_{\ba}\circ \varphi_{\varphi_{\ba}(\bb)}(\bz), \bb\ra_{\cK}|^2}
\cdot \frac{1-\la \ba,\ba\ra_{\cK}}{|1-\la \varphi_{\varphi_{\ba}(\bb)}(\bz),\ba\ra_{\cK}|^2}\\[0.3cm]
& \hspace{1cm}\times \frac{1-\la \varphi_{\ba}(\bb),\varphi_{\ba}(\bb)\ra_{\cK}}{|1-\la \bz, \varphi_{\ba}(\bb)\ra_{\cK}|^2}\cdot(1-\la \bz,\bz \ra_{\cK})\\[0.3cm]
&=\dfrac{(1-\la\bb,\bb\ra_{\cK})|1-\la\bz,\varphi_{\varphi_{\ba}(\bb)}(\ba)\ra_{\cK}|^2|1-\la\varphi_{\ba}(\bb),\ba\ra_{\cK}|^2}
{|1-\la\ba,\bb\ra_{\cK}|^2|1-\la \varphi_{\ba}(\bb),\varphi_{\ba}(\bb)\ra_{\cK}|^2} \\[0.3cm]
& \hspace{1cm}\times \dfrac{(1-\la\ba,\ba\ra_{\cK})|1-\la\bz,\varphi_{\ba}(\bb)\ra_{\cK}|^2}{|1-\la\varphi_{\ba}(\bb),\ba\ra_{\cK}|^2|1-\la\bz,\varphi_{\varphi_{\ba}(\bb)}(\ba)\ra_{\cK}|^2}\\[0.3cm]
& \hspace{2cm}\times \frac{1-\la \varphi_{\ba}(\bb),\varphi_{\ba}(\bb)\ra_{\cK}}{|1-\la \bz, \varphi_{\ba}(\bb)\ra_{\cK}|^2}\cdot(1-\la \bz,\bz \ra_{\cK})\\[0.3cm]
&=\dfrac{(1-\la\ba,\ba\ra_{\cK})(1-\la\bb,\bb\ra_{\cK})}{|1-\la\ba,\bb\ra_{\cK}|^2(1-\la \varphi_{\ba}(\bb), \varphi_{\ba}(\bb)\ra_{\cK})}\cdot(1-\la \bz,\bz \ra_{\cK})\\[0.3cm]
&=1-\la \bz,\bz \ra_{\cK}.
\end{align*}
The above calculations work if $\displaystyle \bz\in \bigcap_{i=1}^n\Omega_{\ba_i}$, where $\ba_1,\ldots, \ba_n$ are  suitable finite nonneutral vectors (e.g., $\ba_1=\ba, \ba_2=\varphi_{\ba}(\bb), \ba_3=\varphi_{\varphi_{\ba}(\bb)}(\ba)$).
Since each $\{\bz\in\Omega:\la\bz, \ba_i\ra_{\cK}=1\}$ is a nowhere dense set, so is $\displaystyle\bigcup_{i=1}^n\{\bz\in\Omega:\la\bz, \ba_i\ra_{\cK}=1\}$. Therefore, $\displaystyle \bigcap_{i=1}^n\Omega_{\ba_i}$ is dense in $\Omega.$
By the continuity of inner product and matrix of finite dimension, 
for any vector $\bz$ in $\cK$, we have that 
$\la T\bz,T\bz \ra_{\cK}=\la \bz,\bz \ra_{\cK}$. 
This concludes the proof. 
\end{proof}

\begin{rem}\rm
It is worth mentioning that 
our proof of Theorem \ref{thm:2-2} is applied not only to 
the Bergman kernel but also to other Hardy--Bergman type kernels 
(see also \cite{KDKS}). 
\end{rem}

\begin{rem}\rm 
If $\ba \neq \0$, then $\varphi_{\ba}$ is a rational map which is not defined everywhere on $\Omega$. 
Such maps play important roles in algebraic geometry. 
See Section 5 in Chapter 5 of Cox--Little--O'Shea~\cite{CLO}, 
in particular Definitions 4, 7 and 9. 
\end{rem}

\section{Indefinite kernel}

In this section, we study the structure of the indefinite kernel defined 
on $\Omega\times \Omega$
\[
K(\bz,\blam)=\dfrac{1}{1-\la \bz, \blam\ra_{\cK}}\quad (\bz,\blam \in \Omega).
\] 
The symbol $\B_2$ denotes the open unit ball
\[
\B_2=\left\{(z_1, z_2)^{\top}\in\C^2:|z_1|^2+|z_2|^2<1\right\}.
\]
Trivially, 
$\B_2$
is included in $\Omega$. Further, 
the following formal calculations make sense on 
$\B_2\times \B_2$:

\begin{align*}
K(\bz,\blam)
&=\sum_{n=0}^{\infty}\la \bz, \blam\ra_{\cK}^n\\
&=\sum_{n=0}^{\infty}(z_1\overline{\lam_1}-z_2\overline{\lam_2})^n\\
&=\sum_{n=0}^{\infty}\left\{\sum_{k=0}^n(-1)^k\binom{n}{k}(z_1\overline{\lam_1})^{n-k}(z_2\overline{\lam_2})^k\right\}\\
&=\sum_{n=0}^{\infty}\left\{\sum_{\ell=0}^{\lfloor n/2 \rfloor}\binom{n}{2\ell}(z_1\overline{\lam_1})^{n-2\ell}(z_2\overline{\lam_2})^{2\ell}\right\}
-\sum_{n=1}^{\infty}\left\{\sum_{\ell=1}^{\lceil n/2 \rceil}\binom{n}{2\ell-1}(z_1\overline{\lam_1})^{n-2\ell+1}(z_2\overline{\lam_2})^{2\ell-1}\right\}. 
\end{align*}
Indeed, if $|z_1|^2+|z_2|^2, |\lambda_1|^2+|\lambda_2|^2\le r<1$,
\begin{align*}
\sum_{n=0}^{\infty}(|z_1\overline{\lambda_1}|+|z_2\overline{\lambda_2}|)^n&\le\sum_{n=0}^{\infty}(|z_1|^2+|z_2|^2)^{n/2}(|\lambda_1|^2+|\lambda_2|^2)^{n/2}\le \sum_{n=0}^{\infty}r^n<\infty.
\end{align*}
Setting 
\[
K_+(\bz,\blam)=\sum_{n=0}^{\infty}\left\{\sum_{\ell=0}^{\lfloor n/2 \rfloor}\binom{n}{2\ell}(z_1\overline{\lam_1})^{n-2\ell}(z_2\overline{\lam_2})^{2\ell}\right\}
\]
and
\[
K_-(\bz,\blam)=\sum_{n=1}^{\infty}\left\{\sum_{\ell=1}^{\lceil n/2 \rceil}\binom{n}{2\ell-1}(z_1\overline{\lam_1})^{n-2\ell+1}(z_2\overline{\lam_2})^{2\ell-1}\right\},
\]
these $K_+$ and $K_-$ are positive semi-definite kernels defined on 
$\B_2\times \B_2$.
Let $\cH_+$ and $\cH_-$ be the reproducing kernel Hilbert spaces generated by $K_+$ and $K_-$, respectively. 
Further, we consider the linear map $\Delta$ defined as 
\[
\Delta: (F_+,F_-)
\to F_+-F_-\quad (F_+\in \cH_+, F_-\in \cH_-).
\]
Note that representation $F=F_+-F_-$ is unique, because 
any function in $\cH_+$ is a power series which has no odd terms of $z_2$, 
but every function in $\cH_-$ has odd terms of $z_2$. 
Next, 
we define an inner product on 
$\operatorname{ran} \Delta$, the range space of $\Delta$, as follows: 
\[
\la \Delta(F_+,F_-),\Delta(G_+,G_-) \ra_{\Delta}=\la F_+,G_+ \ra_{\cH_+}-\la F_-,G_- \ra_{\cH_-}
\]
for any $F_+,G_+$ in $\cH_+$ and any $F_-,G_-$ in $\cH_-$. 

Let $\cK_K(\B_2)$ denote the linear space $\ran \Delta$ endowed with the inner product $\la \cdot ,\cdot\ra_{\Delta}$. 
We see that the Bergman space $L_a^2(\D)$ is isometrically embedded into $\cK_K(\B_2)$ as a Kre\u{\i}n space in the following sense (cf. Theorems 7.31 and 8.2 in Agler-McCarthy~\cite{AM}). We fix a sufficiently small positive number $r$ so that $\Phi(\lambda)$ falls in $\B_2$ for any $\lambda\in r\D$. We note that $L_a^2(\D)|_{r\D}=\{f|_{r\D}:f\in L_a^2(\D)\}$ is a reproducing kernel Hilbert space on $r\D$ with kernel $k|_{r\D\times r\D}$. Then, 
\[
L_a^2(\D)\ominus I(r\D)\rightarrow L_a^2(\D)|_{r\D},\;\;\;\; f\mapsto f|_{r\D}
\]
is a unitary operator, where $I(r\D)=\{f\in L_a^2(\D):f|_{r\D}=0\}.$ Since $r\D$ is open and $L_a^2(\D)$ is a holomorphic Hilbert function space, we have $I(r\D)=0$. Thus,
\[
L_a^2(\D)\rightarrow L_a^2(\D)|_{r\D},\;\;\;\; f\mapsto f|_{r\D}
\]
is a unitary operator. Moreover, we can prove the following result:

\begin{thm}
$\cK_K(\B_2)$ is a reproducing kernel Kre\u{\i}n space on $\B_2$. 
Moreover, $L_a^2(\D)|_{r\D}$ is isometrically embedded into $\cK_K(\B_2)$. More precisely, there exists a linear operator $V:L_a^2(\D)|_{r\D}\rightarrow \cK_K(\B_2)$ such that
$(1)$ $Vk_{\lambda}=K_{\Phi(\lambda)}$ $(\lambda\in r\D)$ and $(2)$ $\la Vf, Vf\ra_{\Delta}=\la f, f\ra$ $(f\in L_a^2(\D)|_{r\D})$ hold.
\end{thm}

\begin{proof}
By the definition of the inner product $\la \cdot ,\cdot \ra_{\Delta}$, 
$\cK_K(\B_2)=(\ran \Delta, \la \cdot ,\cdot\ra_{\Delta})$ is a Kre\u{\i}n space.  
Let $\blam$ be any point in $\B_2$. 
Setting $K_{\blam}(\bz)=K(\bz,\blam)$ and $K_{\pm, \blam}(\bz)=K_{\pm}(\bz,\blam)$, we have 
\[
\la F, K_{\blam} \ra_{\Delta}=\la F_+,K_{+,\blam} \ra_{\cH_+}-\la F_-,K_{-,\blam} \ra_{\cH_-}=F_+(\blam)-F_-(\blam)=F(\blam)
\]
for any $F=\Delta(F_+,F_-)$. 
Hence, $K_{\blam}$ is the reproducing kernel of $\cK_K(\B_2)$ at $\blam$. 
Since $K(\Phi(z),\Phi(w))=k(z,w)$,  
it is easy to see that the map $k_{\lam}\mapsto K_{\Phi(\lam)}$ $(\lambda\in r\D)$ 
extends to an isometric linear embedding of $L_a^2(\D)|_{r\D}$ into $\cK_K(\B_2)$.  
This concludes the proof.
\end{proof}

\begin{rem}\rm 
Our 
$\cK_K(\B_2)$
gives a concrete example to Theorem 3.1 in Alpay~\cite{Alpay}. 
\end{rem}

\section{Pick matrices}

In this section, we shall discuss structure of Pick matrices. 
It is well known that the Bergman kernel does not have the two-point Pick property
(see Example 5.17 in Agler--McCarthy~\cite{AM}). However, 
the following modified two-point interpolation theorem holds. 

\begin{lem}\label{thm:1-4}
Let $\lam_1$ and $\lam_2$ be distinct two points in $\D$, and let $\w_1$ and $\w_2$ be arbitrary two points in $\D$.  
Then, $M_{\Phi}(\lam_1,\lam_2;\w_1,\w_2)$ is positive semi-definite if and only if 
there exists a function $f$ in $\cS$ such that $f(\lam_j)=\w_j$ for any $j=1,2$. 
\end{lem}

\begin{proof}
We set $M_{\Phi}=M_{\Phi}(\lam_1,\lam_2;\w_1,\w_2)$. 
Suppose that $M_{\Phi}$ is positive semi-definite. 
Then, by Proposition \ref{lem:1-2}, we have that $\rho(\w_1,\w_2) \leq \rho (\lam_1,\lam_2)$. 
Moreover, by the same calculation as in the proof of Proposition \ref{prop:1-1}, we have that
\[
\left|\dfrac{\w_1-\w_2}{1-\w_1\overline{\w_2}}\right|
\leq\left|\dfrac{\lam_1-\lam_2}{1-\lam_1\overline{\lam_2}}\right|.
\]
This concludes the only if part.  
Since $M_{\Phi}$ is the square of a Pick matrix with respect to the Schur product, 
the if part is trivial by Schur's theorem. 
\end{proof}

A matrix $T$ is called a contraction on $\cK$ if $\la T\bz,T\bz \ra_{\cK}\leq \la\bz,\bz \ra_{\cK}$ for any $\bz$ in $\cK$ 
(see p.\ 150 in Dritschel--Rovnyak~\cite{DR})
We note that $T$ is a contraction on $\cK$ if and only if 
\begin{equation}\label{eq:3-1}
\begin{pmatrix}
\la T\blam_i, T\blam_j \ra_{\cK}
\end{pmatrix}_{1\leq i,j \leq n}
\leq 
\begin{pmatrix}
\la \blam_i,\blam_j \ra_{\cK}
\end{pmatrix}_{1\leq i,j \leq n}
\end{equation}
for any $n$ in $\mathbb N$ and $\blam_1,\ldots,\blam_n$ in $\cK$.  
Further, $\cR(\Omega)$ denotes the set of all rational self-maps of $\Omega$. 
A typical example of an element of $\cR(\Omega)$ is $\varphi_{\ba}$ discussed 
in Section 3. 
We define the following class: 
\[
\cR\cS_{\Phi}(\Omega)=\left\{F\in \cR(\Omega): \dfrac{1-\la F\circ \Phi(z), F\circ \Phi(\lam)\ra_{\cK}}
{1-\la \Phi(z),\Phi(\lam) \ra_{\cK}}\ \text{is positive semi-definite in $\D^2$}\right\}. 
\]
\begin{ex}\label{ex:5-2}\rm
Let $T$ be any contraction on $\cK$. 
Then, $T\varphi_{\ba}$ belongs to $\cR\cS_{\Phi}(\Omega)$. 
Indeed, 
it follows from 
(\ref{eq:3-1}), Schur's theorem and (ii) in Proposition \ref{prop:1-3} that 
\begin{align*}
\begin{pmatrix}\dfrac{1-\la T\varphi_{\ba}\circ \Phi(\lam_i), T\varphi_{\ba}\circ \Phi(\lam_j)\ra_{\cK}}
{1-\la \Phi(\lam_i),\Phi(\lam_j) \ra_{\cK}}
\end{pmatrix}_{1\leq i,j \leq n}
&\geq\begin{pmatrix}\dfrac{(1-\la \varphi_{\ba}\circ \Phi(\lam_i),\varphi_{\ba}\circ \Phi(\lam_j) \ra_{\cK})}
{1-\la \Phi(\lam_i),\Phi(\lam_j) \ra_{\cK}}
\end{pmatrix}_{1\leq i,j \leq n}\\
&=\begin{pmatrix}\dfrac{(1-\la \ba,\ba \ra_{\cK})}{(1-\la \Phi(\lam_i), \ba\ra_{\cK})(1-\la \ba,\Phi(\lam_j) \ra_{\cK})}
\end{pmatrix}_{1\leq i,j \leq n}.
\end{align*}
\end{ex}

The following is the main result of this paper. 
\begin{thm}\label{thm:4-1}
Let $\lam_1$ and $\lam_2$ be distinct two points in $\D$, and let $\w_1$ and $\w_2$ be points in $\D$. 
Then, 
the following conditions are equivalent.
\begin{enumerate}
\item $M_{\Phi}(\lam_1,\lam_2;\w_1,\w_2)$ is positive semi-definite.   
\item $P(\lam_1,\lam_2;\w_1,\w_2)$ is positive semi-definite, where $P(\lam_1,\lam_2;\w_1,\w_2)$ denotes the Pick matrix with respect to $(\lam_1,\lam_2)$ and $(\w_1, \w_2)$. 
\item $\rho(\w_1,\w_2)\leq \rho(\lam_1,\lam_2)$.
\item There exists a map $F$ in $\cR\cS_{\Phi}(\Omega)$  
such that $F(\Phi(\lam_j))=\Phi(\w_j)$ for any $j=1,2$. 
\end{enumerate}
\end{thm}

\begin{proof}
By Proposition \ref{lem:1-2} and Lemma \ref{thm:1-4}, (i), (ii) and (iii) are equivalent. 
We shall show that (i) and (iv) are equivalent. 
Suppose that $M_{\Phi}(\lam_1,\lam_2;\w_1,\w_2)$ is positive semi-definite. 
We set $\blam=\varphi_{\Phi(\lam_2)}(\Phi(\lam_1))$ and $\bw=\varphi_{\Phi(\w_2)}(\Phi(\w_1))$. 
Then, by (vi) of Proposition \ref{prop:1-3} and Proposition \ref{lem:1-2}, we have that 
\begin{align*}
\la \bw,\bw\ra_{\cK}
&=\la \varphi_{\Phi(\w_2)}(\Phi(\w_1)), \varphi_{\Phi(\w_2)}(\Phi(\w_1)) \ra_{\cK}\\
&=\rho(\w_1,\w_2)^2\\
&\leq \rho(\lam_1,\lam_2)^2\\
&=\la \varphi_{\Phi(\lam_2)}(\Phi(\lam_1)), \varphi_{\Phi(\lam_2)}(\Phi(\lam_1)) \ra_{\cK}\\
&=\la \blam,\blam \ra_{\cK}.
\end{align*}
In what follows, 
we write $[\bx]^{[\perp]}=\{\by\in \cK:\la \bx,\by \ra_{\cK}=0\}$ for any vector $\bx$ in $\cK$. 
We set 
\[
T\bz=\dfrac{1}{\la\blam,\blam\ra_{\cK}}\la \bz,\blam\ra_{\cK}\bw-\la \bz,\bmu \ra_{\cK} \bnu,
\]
where $\bmu$ and $\bnu$ are chosen so that 
$\bmu\in [\blam]^{[\perp]}$, $\bnu\in [\bw]^{[\perp]}$ and $\la \bmu, \bmu\ra_{\cK}, \la \bnu, \bnu\ra_{\cK}=-1$.
Then, $T\blam=\bw$ and $T$ is a contraction on $\cK$. 
Setting $F=\varphi_{\Phi(\w_2)}\circ T \circ \varphi_{\Phi(\lam_2)}$, we have that $F(\Phi(\lam_j))=\Phi(\w_j)$ for any $j=1,2$. 
Further, it follows from (ii) of Proposition \ref{prop:1-3} and (\ref{eq:3-1}) that
\[
\dfrac{1-\la F\circ \Phi(z), F\circ \Phi(\lam)\ra_{\cK}}
{1-\la \Phi(z),\Phi(\lam) \ra_{\cK}}
=
\dfrac{1-\la 
\varphi_{\Phi(\w_2)}\circ T \circ \varphi_{\Phi(\lam_2)}\circ\Phi(z), \varphi_{\Phi(\w_2)}\circ T \circ \varphi_{\Phi(\lam_2)} \circ\Phi(\lam)\ra_{\cK}}
{1-\la \Phi(z),\Phi(\lam) \ra_{\cK}}
\]
is positive semi-definite (see also Example \ref{ex:5-2}). 
Hence, we conclude that (i) implies (iv). 
The converse is trivial by the definition of $\cR\cS_{\Phi}(\Omega)$. 
\end{proof}

Diligent readers would have found already that 
$M_{\Phi}(\lam_1,\lam_2;\mu_1,\mu_2)$ is the Schur product of two Pick matrices 
$P(\lam_1,\lam_2;\mu_1,\mu_2)$. 
Let $P$ be a two-by-two Pick matrix. Then, 
as a corollary of Theorem \ref{thm:4-1}, we obtain that 
$P\odot P$ is positive semi-definite if and only if so is $P$, 
where $\odot$ denotes the Schur product.
However, the positive semi-definiteness of $P\odot P$ of a three-by-three Pick matrix $P$ does not imply that of $P$.
To see this, we consider the following example.
\begin{ex}\rm 
We set $\lam_1=2/3, \lam_2=3/4, \w_1=1/3, \w_2=1/4$. 
Then, the Pick matrix 
with respect to $\lam_1\mapsto \w_1$, $\lam_2\mapsto \w_2$, $0\mapsto 0$ is
\[
\begin{pmatrix}
\dfrac{1-|\w_1|^2}{1-|\lam_1|^2}&\dfrac{1-\w_1\overline{\w_2}}{1-\lam_1\overline{\lam_2}}&1\\[0.5cm]
\dfrac{1-\w_2\overline{\w_1}}{1-\lam_2\overline{\lam_1}}&\dfrac{1-|\w_2|^2}{1-|\lam_2|^2}&1\\[0.5cm]
1&1&1
\end{pmatrix}
=\begin{pmatrix}
\dfrac{8}{5}&\dfrac{11}{6}&1\\[0.5cm]
\dfrac{11}{6}&\dfrac{15}{7}&1\\[0.5cm]
1&1&1
\end{pmatrix},
\]
which is not positive semi-definite. In fact,
\[
\det
\begin{pmatrix}
\dfrac{8}{5}&\dfrac{11}{6}&1\\[0.5cm]
\dfrac{11}{6}&\dfrac{15}{7}&1\\[0.5cm]
1&1&1
\end{pmatrix}
=-\frac{11}{1260}<0.
\]
However, 
\begin{equation*}
\displaystyle
\begin{pmatrix}
\left(\dfrac{1-|\w_1|^2}{1-|\lam_1|^2}\right)^2&\left(\dfrac{1-\w_1\overline{\w_2}}{1-\lam_1\overline{\lam_2}}\right)^2&1\\[0.5cm]
\left(\dfrac{1-\w_2\overline{\w_1}}{1-\lam_2\overline{\lam_1}}\right)^2&\left(\dfrac{1-|\w_2|^2}{1-|\lam_2|^2}\right)^2&1\\[0.5cm]
1&1&1
\end{pmatrix}=
\begin{pmatrix}
\dfrac{64}{25}&\dfrac{121}{36}&1\\[0.5cm]
\dfrac{121}{36}&\dfrac{225}{49}&1\\[0.5cm]
1&1&1
\end{pmatrix}
\end{equation*}
is positive definite. 
Indeed, 
\[
\det 
\begin{pmatrix}
\dfrac{64}{25}&\dfrac{121}{36}\\[0.5cm]
\dfrac{121}{36}&\dfrac{225}{49}
\end{pmatrix}
=\dfrac{29087}{63504}>0\quad \mbox{and}\quad
\det
\begin{pmatrix}
\dfrac{64}{25}&\dfrac{121}{36}&1\\[0.5cm]
\dfrac{121}{36}&\dfrac{225}{49}&1\\[0.5cm]
1&1&1
\end{pmatrix}
=\dfrac{45119}{1587600}>0.
\]
\end{ex}

\begin{acknowledgment}\rm
The authors would like to thank 
Professor Jaydeb Sarkar (Indian Statistical Institute) 
for his suggestion which led us to this research. 
The third author would like to thank Professor Humihiko Watanabe (National Defense Academy of Japan) 
for his advice on algebraic geometry. 
This work was supported by JSPS KAKENHI Grant Number JP24K06771.
\end{acknowledgment}

\end{document}